\documentclass[12pt]{amsart}
 
\pdfoutput=1
\usepackage[margin=1in]{geometry}
\usepackage{amsmath,amsthm,amssymb,amsrefs,bbm,color,esvect,float,graphicx,mathrsfs}
\usepackage{epstopdf}
\AppendGraphicsExtensions{.tif}

\newenvironment{theorem}[2][Theorem]{\begin{trivlist}
\item[\hskip \labelsep {\bfseries #1}\hskip \labelsep {\bfseries #2.}]}{\end{trivlist}}
\newenvironment{lemma}[2][Lemma]{\begin{trivlist}
\item[\hskip \labelsep {\bfseries #1}\hskip \labelsep {\bfseries #2.}]}{\end{trivlist}}

\newenvironment{remark}[2][Remark]{\begin{trivlist}
\item[\hskip \labelsep {\bfseries #1}\hskip \labelsep {\bfseries #2.}]}{\end{trivlist}}
\theoremstyle{definition}

\begin{document}

\title[Weak-type estimates for Calder\'on-Zygmund operators]{A different approach to endpoint weak-type estimates for Calder\'on-Zygmund operators}
\author{Cody B. Stockdale}
\address{Cody B. Stockdale, Department of Mathematics, Washington University in St. Louis, One Brookings Drive, St. Louis, MO, 63130, USA}
\email{codystockdale@wustl.edu}

\maketitle
\vspace{-.25in}
\begin{abstract}
We present a new proof of the classical weak-type $(1,1)$ estimate for Calder\'on-Zygmund operators. This proof is inspired by ideas of Nazarov, Treil, and Volberg that address the non-doubling setting. An application to a weighted weak-type inequality is also given.
\end{abstract}

\section{Introduction}
Let $T$ denote a Calder\'on-Zygmund singular integral operator. For a measurable set $A \subseteq \mathbb{R}^n$, the quantity $|A|$ denotes the Lebesgue measure of $A$. The ball with center $x\in \mathbb{R}^n$ and radius $r>0$ is denoted by $B(x,r)$, while the cube with center $x\in \mathbb{R}^n$ and side length $r>0$ is denoted by $Q(x,r)$. For a cube $Q$, the notation $rQ$ describes the cube with the same center as $Q$ and with side length $r$ times the length of $Q$. We use the notation $A \lesssim B$ to mean there exists $C_{n,T}>0$, possibly depending on $n$ or $T$, such that $A \leq C_{n,T}B$.

It is well-known that Calder\'on-Zygmund operators are bounded on $L^p(\mathbb{R}^n)$ for $1<p<\infty$ and are generally unbounded on $L^1(\mathbb{R}^n)$. For the endpoint case $p=1$, we instead have the following fundamental result known as the \emph{weak-type $(1,1)$ property}. 
\begin{theorem}{1}
Any Calder\'on-Zygmund operator $T$ satisfies 
$$
    \|Tf\|_{L^{1,\infty}(\mathbb{R}^n)}:=\sup_{\lambda>0}\lambda|\{|Tf|>\lambda\}|\lesssim \|f\|_{L^1(\mathbb{R}^n)}
$$
for all $f \in L^1(\mathbb{R}^n)$.
\end{theorem}
The boundedness of $T$ on $L^p(\mathbb{R}^n)$ for $1<p<\infty$ follows from Theorem 1 and the Marcin-\\
kiewicz interpolation theorem, see \cites{Grafakos1,Grafakos2}.

Theorem 1 was originally proved using the Calder\'on-Zygmund decomposition. This decomposition relies on the doubling property, which for a Borel measure $\mu$ on a space $X$ means there exists $C_{\mu}>0$ such that 
$$
    \mu(B(x,2r))\leq C_{\mu}\mu(B(x,r))
$$ 
for all $x \in X$ and all $r>0$. In \cite{NTV1998}, Nazarov, Treil, and Volberg recovered the basic theory of Calder\'on-Zygmund operators in a setting where the doubling property of the underlying measure is replaced by the following polynomial growth condition: there exists $n, C_{\mu}>0$ such that 
$$
    \mu(B(x,r))\leq C_{\mu}r^n
$$
for all $x \in X$ and all $r>0$. In particular, they proved the weak-type $(1,1)$ inequality. Since the doubling property is not available in this setting, their proof avoids the Calder\'on-Zygmund decomposition.

Unlike the setting of \cite{NTV1998}, the Euclidean setting of Theorem 1 allows the doubling property. We use the doubling property to obtain the main result of Section 3 -- a new simple proof of Theorem 1 motivated by the ideas of Nazarov, Treil, and Volberg. 

Following \cite{NTV1998}, the weak-type $(1,1)$ property reduces to proving
$$
	\|T\nu\|_{L^{1,\infty}(\mathbb{R}^n)}\lesssim \|\nu\|,
$$
where $\nu$ is a linear combination of point-mass measures and $\|\nu\|$ denotes the total variation of $\nu$. This inequality involves approximating $\nu$ by appropriately constructed Borel sets, and then it is left to estimate a final term using the size condition of the Calder\'on-Zygmund kernel, a duality trick involving the adjoint of $T$, and control of the maximal truncation operator. Using the doubling property, the weak-type estimate on point-mass measures, the size condition of the kernel, duality, and the maximal truncation operator control are no longer needed. Instead, we obtain cancellation by directly approximating with explicitly constructed Borel sets and, due to the doubling property of Lebesgue measure, we may easily bound the remaining term. 

The Nazarov-Treil-Volberg technique can be adapted to handle more general situations; see, for example, the proof of the weak-type $\left(1,\ldots,1;\frac{1}{m}\right)$ estimate for $m$-multilinear Calder\'on-Zygmund operators given by the author and Wick in \cite{SW2019}. Another application is given in Section 4, where a weak-type $(1,1)$ inequality involving $A_p$ weights is proved. A locally integrable function $w$ on $\mathbb{R}^n$ is called a \emph{weight} if $w(x)>0$ for almost every $x \in \mathbb{R}^n$. For $1< p < \infty$, the class $A_p$ consists of all weights $w$ satisfying the $A_p$ condition 
$$
    [w]_{A_p}:=\sup_{Q}\left(\frac{1}{|Q|}\int_{Q}w(x)dx\right)\left(\frac{1}{|Q|}\int_{Q}w(x)^{1-p'}dx\right)^{p-1}<\infty,
$$
where $p'$ is the H\"older conjugate of $p$ and the supremum is taken over all cubes $Q \subseteq \mathbb{R}^n$. %When $p=1$, the quantity $\left(\frac{1}{|Q|}\int_{Q}w(x)^{1-p'}dx\right)^{p-1}$ is interpreted as $\left(\inf_Q w\right)^{-1}$. 
For a weight $w$ and $A \subseteq \mathbb{R}^n$, the quantity $w(A)$ represents $\int_A w(x)\,dx$. Notice that if $w \in A_p$, then $w(x)dx$ is a doubling measure with 
$$
    w(Q(x,ar))\leq a^{np}[w]_{A_p}w(Q(x,r))
$$
for all $x \in \mathbb{R}^n$, $r>0$, and $a>1$; see \cites{Grafakos1,Grafakos2}. \begin{theorem}{2}
If $1< p<\infty$ and $w \in A_p$, then 
$$
    \|T(fw)w^{-1}\|_{L^{1,\infty}(w)} \lesssim [w]_{A_p}\max\{p,\log(e+[w]_{A_p})\}\|f\|_{L^1(w)}
$$
for all $f \in L^1(w)$.
\end{theorem}
Theorem 2 was proved using the Calder\'on-Zygmund decomposition by Ombrosi, P\'erez, and Recchi in \cite{OPR2016}. A new proof of Theorem 2 is given in Section 4. Similar weighted weak-type inequalities appear in various forms, see \cites{CUMP2005, LOP2019, CRR2018, OP2016}.

We compare the Calder\'on-Zygmund decomposition proof of Theorem 1 and the proof given in Section 3. To prove the weak-type $(1,1)$ property, one shows $$|\{|Tf|>\lambda\}|\lesssim \frac{1}{\lambda}\|f\|_{L^1(\mathbb{R}^n)}$$ for all $\lambda>0$ and all $f\in L^1(\mathbb{R}^n)$. Both techniques involve decomposing $f$ into summands, $$
    f=g+b=g+\sum_{i=1}^{\infty}b_i,
$$ 
where $g$ is ``good'' and $b$ is ``bad,'' and then controlling 
$$
    |\{|Tf|>\lambda\}|\leq \left|\left\{|Tg|>\frac{\lambda}{2}\right\}\right|+\left|\left\{|Tb|>\frac{\lambda}{2}\right\}\right|.
$$ 
In both arguments, the term involving $g$ is handled by using Chebyshev's inequality, the boundedness of $T$ on $L^2(\mathbb{R}^n)$, and the $L^{\infty}(\mathbb{R}^n)$ control of $g$. The terms involving $b$ are estimated differently.

Much of the effort in the Calder\'on-Zygmund decomposition method is spent in carefully decomposing $f$ into its ``good'' and ``bad'' parts so that the functions $b_{i}$ have mean value zero and have useful $L^1(\mathbb{R}^n)$ control. This decomposition typically involves averages of $f$ and the use of the doubling property. After defining an exceptional set, $\Omega^*$, in terms of the supports of the $b_i$, one estimates 
$$
    \left|\left\{|Tb|>\frac{\lambda}{2}\right\}\right| \leq \left|\Omega^*\right|+ \left|\left\{x \in \mathbb{R}^n\setminus \Omega^*: |Tb(x)|>\frac{\lambda}{2}\right\}\right|.
$$ 
The first term is controlled due to properties of the Calder\'on-Zygmund decomposition and the doubling property. The final term is controlled using cancellation of the $b_{i}$, the smoothness assumption of the kernel of $T$, and the $L^1(\mathbb{R}^n)$ control of the $b_i$.

Using ideas from \cite{NTV1998}, the decomposition of $f$ into its ``good'' and ``bad'' parts is more direct. The exceptional set is defined explicitly as
$$
    \Omega:=\{|f|>\lambda\},
$$
then $g$ and $b$ are defined by
$$
    g:=f\mathbbm{1}_{\mathbb{R}^n\setminus \Omega}\quad\quad\text{and}\quad\quad b:=f\mathbbm{1}_{\Omega}.
$$
The $b_i$ are defined by applying a Whitney decomposition to write $\Omega$ as a disjoint union of cubes and restricting $b$ to each cube. To introduce cancellation in the $b_i$, Borel sets, $E_i$, of appropriate measure are constructed around the center of $\text{supp}(b_i)$, and a related set, $E^*$, is included in the exceptional set. Adding and subtracting $\lambda T(\mathbbm{1}_E)$, where $E$ is the union of the $E_i$, one estimates 
$$
   \,\,\, \left|\left\{|Tb|>\frac{\lambda}{2}\right\}\right|\leq \left|\Omega\cup E^* \right| + \left|\left\{x\in\mathbb{R}^n\setminus(\Omega\cup E^*): |T(b-\lambda\mathbbm{1}_E)(x)|>\frac{\lambda}{4}\right\}\right|
$$
$$
    +\left|\left\{|T(\mathbbm{1}_E)|>\frac{1}{4}\right\}\right|.
$$ 
The first term is handled with Chebyshev's inequality and the doubling property. The second term is controlled since $\mathbbm{1}_E$ introduces cancellation that allows for the use of the smoothness assumption of the kernel. The final term is controlled in a way similar to the the term involving the ``good'' function.

Our proof has some benefits over the Calder\'on-Zygmund decomposition technique. For example, the decomposition used to write $f=g+b$ in this argument does not involve studying averages of $f$ or the doubling property. The doubling property is only used later in the proof to gain control over $|E^*|$. Also, this proof shows that $L^1$ control of the $b_i$ is not necessary for the weak-type $(1,1)$ estimate and demonstrates a measure-theoretic method to gain cancellation in the $b_i$. 

Relevant definitions and lemmas are described in Section 2. Section 3 includes the new proof of Theorem 1. Section 4 contains the application to Theorem 2.

I would like to thank Brett Wick for conversations regarding this article.

%%%%%%%%%%%%%%%%%%%%%%%%%%%%%%%%%%%%%%%%%%%%%%%%%%%%%%%%%%%%%%%%%%%%%%%%%%%%

\section{Preliminaries}
We say $K:(\mathbb{R}^{n} \times \mathbb{R}^n)\setminus\{(x,y):x=y\} \rightarrow \mathbb{C}$ is a \emph{Calder\'on-Zygmund kernel} if
\begin{enumerate}
\item (\emph{size}) \[|K(x,y)| \lesssim \frac{1}{|x-y|^{n}}\]
for all $x,y\in \mathbb{R}^n$ with $x\neq y$,
\item (\emph{smoothness}) there exists $\delta >0$ such that 
$$
    |K(x,y)-K(x',y)| \lesssim \frac{|x-x'|^{\delta}}{|x-y|^{n+\delta}}
$$
whenever $\displaystyle |x-x'| \leq \frac{1}{2}|x-y|$, and 
$$
    |K(x,y)-K(x,y')|\lesssim \frac{|y-y'|^{\delta}}{|x-y|^{n+\delta}}
$$
whenever $\displaystyle |y-y'| \leq \frac{1}{2}|x-y|$.
\end{enumerate}
We say a a linear operator $T$ is a \emph{Calder\'on-Zygmund operator with kernel $K$} if $K$ is a Calder\'on-Zygmund kernel, $T$ extends to a bounded operator on $L^2(\mathbb{R}^n)$, and $T$ is given by
$$
    Tf(x) = \int_{\mathbb{R}^n} K(x,y)f(y)dy
$$ 
for smooth compactly supported $f$ and almost every $x \in \mathbb{R}^n\setminus\text{supp}(f)$.

\begin{lemma}{1}
If $f\in L^1(\mathbb{R}^n)$ is supported on $Q(x,r)$ and $\int_{Q(x,r)}f(y)dy=0$ for some $x\in\mathbb{R}^n$ and $r>0$, then 
$$
    \int_{\mathbb{R}^n\setminus Q(x,2\sqrt{n}r)}|Tf(y)|dy\lesssim \|f\|_{L^1(\mathbb{R}^n)}.
$$
\end{lemma}

\begin{proof}
First, notice that since $\int_{Q(x,r)}f(y)dy=0$ and $\text{supp}(f)\subseteq Q(x,r)$, 
$$
    |Tf(y)| = \left|\int_{Q(x,r)}K(y,z)f(z)dz\right|=\left|\int_{Q(x,r)} (K(y,z)-K(y,x))f(z)\,dz\right|.
$$ 
Therefore, using Fubini's theorem and the smoothness estimate of $K$, we see
\begin{align*}
\int_{\mathbb{R}^n\setminus Q(x,2\sqrt{n}r)}|Tf(y)|\,dy &\leq \int_{\mathbb{R}^n\setminus Q(x,2\sqrt{n}r)}\int_{Q(x,r)}|K(y,z)-K(y,x)||f(z)|\,dzdy\\
&=\int_{Q(x,r)}|f(z)|\int_{\mathbb{R}^n\setminus Q(x,2\sqrt{n}r)}|K(y,z)-K(y,x)|\,dydz\\
&\leq\int_{Q(x,r)}|f(z)|\int_{d(x,y)\ge 2d(x,z)} |K(y,z)-K(y,x)|\,dydz\\
&\lesssim \int_{Q(x,r)}|f(z)|\int_{d(x,y)\ge 2d(x,z)} \frac{|x-z|^{\delta}}{|x-y|^{n+\delta}}\,dydz\\
&\lesssim \|f\|_{L^1(\mathbb{R}^n)}.
\end{align*}
\end{proof}

\begin{lemma}{2}
Let $\mu$ be a doubling measure on $\mathbb{R}^n$ such that 
$$
    \mu(Q(x,ar))\leq C_{\mu,a} \mu(Q(x,r))
$$
for all $x \in \mathbb{R}^n$, $r>0$, and $a >1$. If $N$ is a positive integer, then 
$$
    \mu\left(\bigcup_{j=1}^N Q(x_j,ar_j)\right)\leq C_{\mu,a}\mu\left(\bigcup_{j=1}^N Q(x_j,r_j)\right)
$$
for all $x_1,\ldots, x_N \in \mathbb{R}^n$, $r_1,\ldots,r_N>0$, and $a>1$. %The inequality still holds if the cubes are replaced by balls induced by any $\ell^p$ norm on $\mathbb{R}^n$.
\end{lemma}

\begin{proof}
Reorder the $r_j$ to assume that $r_1\ge r_2\ge \cdots \ge r_N$. Set 
$$
    F_1:=Q(x_1,r_1),\quad F_j:=Q(x_j,r_j)\setminus \bigcup_{k=1}^{j-1}F_k, \quad \text{and} \quad F:=\bigcup_{j=1}^N F_j=\bigcup_{j=1}^N Q(x_j,r_j).
$$
Similarly, set
$$
    F_1^*:=Q(x_1,ar_1),\quad F_j^*:= Q(x_j,ar_j)\setminus \bigcup_{k=1}^{j-1}F_k^*, \quad\text{and}\quad F^*:=\bigcup_{j=1}^N F_j^*=\bigcup_{j=1}^NQ(x_j,ar_j).
$$

For each $j \in \{1,2,\ldots,N\}$, we claim that $F_j^*$ is a dilation of $F_j$ in the sense that
$$
    F_j^*\subseteq\widetilde{F}_j,
$$
where $\widetilde{F}_j:=\{a(y-x_j)+x_j:y\in F_j\}$. Note that $\mu(\widetilde{F}_j)\leq C_{\mu,a}\mu(F_j)$ since $\widetilde{F}_j$ is obtained from $F_j$ by composing a translation and a dilation by a factor of $a$. Assuming the claim, we conclude 
$$
    \mu(F^*)\leq \sum_{j=1}^{N}\mu(F_j^*)=\sum_{j=1}^N\mu(\widetilde{F}_j)\leq C_{\mu,a}\sum_{j=1}^N\mu(F_j)=C_{\mu,a}\mu(F).
$$

It remains to prove $F_j^*\subseteq\widetilde{F}_j$. Let $x \in F_j^*$. Since $F_j^*\subseteq Q(x_j,ar_j)$, we can write $x=a(y-x_j)+x_j$ for some $y \in Q(x_j,r_j)$. Since $\displaystyle Q(x_j,r_j)\subseteq\bigcup_{k=1}^j F_k$ and the $F_k$ are pairwise disjoint, $y \in F_{k_0}\subseteq Q(x_{k_0},r_{k_0})$ for some distinguished $1\leq k_0 \leq j$. Suppose that $k_0<j$. Since $r_{k_0}\ge r_j$, we have 
$$
    |x-x_{k_0}|_{\infty}=|a(y-x_j)+x_j-x_{k_0}|_{\infty}\leq (a-1)|y-x_j|_{\infty}+|y-x_{k_0}|_{\infty}
$$
$$
    <(a-1)r_j+r_{k_0}\leq ar_{k_0}.
$$
This implies $\displaystyle x \in Q(x_{k_0},ar_{k_0})\subseteq \bigcup_{k=1}^{j-1}F_K^*,$ contradicting the fact that $x\in F_j^*$. Therefore $y \in F_j$, and $x \in \widetilde{F}_j$. 

%Let $x \in \widetilde{F}_j$. By the induction hypothesis, $\displaystyle \bigcup_{k=1}^{j-1}\widetilde{F}_k = \bigcup_{k=1}^{j-1}F_k^*$. Since the $F_k^*$ are pairwise disjoint, $\displaystyle x \not \in \bigcup_{k=1}^{j-1}\widetilde{F}_k=\bigcup_{k=1}^{j-1} F_k^*$, by induction hypothesis. And since $F_j \subseteq Q(x_j,r_j)$, it follows that $x \in Q(x_j,2r_j)$. Therefore $x \in F_j^*$. 
\end{proof}

\begin{remark}{1}
Lemma 2 implies
$$
    \left|\bigcup_{j=1}^N Q(x_j,ar_j)\right|\leq a^n\left|\bigcup_{j=1}^N Q(x_j,r_j)\right|,
$$
and for $w \in A_p$,
$$
    w\left(\bigcup_{j=1}^N Q(x_j,ar_j)\right)\leq a^{np}[w]_{A_p}w\left(\bigcup_{j=1}^N Q(x_j,r_j)\right).
$$
\end{remark}

%%%%%%%%%%%%%%%%%%%%%%%%%%%%%%%%%%%%%%%%%%%%%%%%%%%%%%%%%%%%%%%%%%%%%%%%%%

\section{Unweighted Estimate}
\begin{theorem}{1}
Any Calder\'on-Zygmund operator $T$ satisfies 
$$
    \|Tf\|_{L^{1,\infty}(\mathbb{R}^n)} \lesssim\|f\|_{L^1(\mathbb{R}^n)}
$$
for all $f \in L^1(\mathbb{R}^n)$.
\end{theorem}

%\begin{theorem}{3}
%\label{theorem3}
%If $1<p<\infty$, $w \in A_p$, and $\nu \in \mathcal{M}(\mathbb{R}^n)$ is of the form $\nu = \sum_{j=1}^{N} a_{j}\delta_{x_{j}}$ where $x_{j} \in \mathbb{R}^n$ and $a_{j}\in\mathbb{R}$, then for any $t>0$ $$w(\{|T\nu|w^{-1}>t\}|\leq \frac{B_4}{t}\|\nu\|$$ where $B_4$ depends on $K$, $[w]_{A_p}$, and $n$. 
%\end{theorem}

\begin{proof}[Proof of Theorem 1]
Let $\lambda>0$ be given. We wish to show $$|\{|Tf|>\lambda\}|\lesssim \frac{1}{\lambda}\|f\|_{L^1(\mathbb{R}^n)}.$$ By density, we may assume $f$ is a nonnegative continuous function with compact support. Set 
$$
    \Omega:=\left\{f>\lambda\right\}.
$$ 
Apply a Whitney decomposition to write 
$$
    \Omega=\bigcup_{i=1}^{\infty}Q_i,
$$ 
a disjoint union of dyadic cubes where 
$$
    2\text{diam}(Q_i)\leq d(Q_i,\mathbb{R}^n\setminus \Omega)\leq 8\text{diam}(Q_i).
$$ 
Put 
$$
    g:=f\mathbbm{1}_{\mathbb{R}^n\setminus \Omega}, \quad\quad b:=f\mathbbm{1}_{\Omega}, \quad\quad \text{and} \quad\quad b_i:=f\mathbbm{1}_{Q_i}.
$$
Then 
$$
    f=g+b=g+\sum_{i=1}^{\infty}b_i,
$$ 
where
\begin{itemize}
    \item[(1)] $\|g\|_{L^{\infty}(\mathbb{R}^n)}\leq \lambda$ and $\|g\|_{L^1(\mathbb{R}^n)} \leq \|f\|_{L^1(\mathbb{R}^n)}$,
    \item[(2)] the $b_i$ are supported on pairwise disjoint cubes $Q_i$, where $\sum_{i=1}^{\infty}|Q_i| \leq \frac{1}{\lambda}\|f\|_{L^1(\mathbb{R}^n)}$, and 
    \item[(3)] $\|b\|_{L^1(\mathbb{R}^n)}\leq \|f\|_{L^1(\mathbb{R}^n)}$.
\end{itemize}
Then
$$
    |\{|Tf|>\lambda\}|\leq \left|\left\{|Tg|>\frac{\lambda}{2}\right\}\right|+\left|\left\{|Tb|>\frac{\lambda}{2}\right\}\right|.
$$

To control the first term, use Chebyshev's inequality, the boundedness of $T$ on $L^2(\mathbb{R}^n)$, and property (1) to estimate
\begin{align*}
\left|\left\{|Tg|>\frac{\lambda}{2}\right\}\right|&\lesssim \frac{1}{\lambda^2}\int_{\mathbb{R}^n}|Tg(x)|^{2}dx\\
&\lesssim \frac{1}{\lambda^2}\int_{\mathbb{R}^n}|g(x)|^{2}dx\\
&\leq \frac{1}{\lambda}\int_{\mathbb{R}^n}|g(x)|dx\\
&\leq \frac{1}{\lambda}\|f\|_{L^1(\mathbb{R}^n)}.
\end{align*}

For positive integers $N$, set $b^N:=\sum_{i=1}^{N}b_i$. To control the second term, it suffices to handle $\left|\left\{|Tb^N|>\frac{\lambda}{2}\right\}\right|$ uniformly in $N$. Let $c_i$ denote the center of $Q_i$ and let $a_i:=\int_{Q_i}b_i(x)dx$. Set 
$$
    E_{1}:=Q(c_{1},r_{1}),
$$ 
where $r_{1}>0$ is chosen so that $|E_{1}|=\frac{a_1}{\lambda}$. In general, for $i=2, 3,\ldots,N$, set 
$$
    E_{i}:=Q(c_{i},r_{i})\setminus \bigcup_{k=1}^{i-1}E_{k},
$$ 
where $r_{i}>0$ is chosen so that $|E_{i}|=\frac{a_i}{\lambda}$. Note that such $E_i$ exist since the function $r \mapsto |Q(x,r)|$ is continuous for each $x \in \mathbb{R}^n$. Define 
$$
    E:=\bigcup_{i=1}^N E_i=\bigcup_{i=1}^N Q(c_i,r_i) \quad\quad\text{and}\quad\quad E^*:=\bigcup_{i=1}^NQ(c_i,2\sqrt{n}r_i).
$$
Then 
$$
    \left|\left\{|Tb^N|>\frac{\lambda}{2}\right\}\right|\leq \text{I}+\text{II}+\text{III},
$$
where 
\begin{align*}
    &\text{I}:=|\Omega\cup E^*|,\\
    &\text{II}:=\left|\left\{x\in \mathbb{R}^n\setminus(\Omega\cup E^*):|T(b^N-\lambda\mathbbm{1}_E)(x)|>\frac{\lambda}{4}\right\}\right|,\quad\text{and}\\
    &\text{III}:=\left|\left\{|T(\mathbbm{1}_E)|>\frac{1}{4}\right\}\right|.
\end{align*}

The control of $\text{I}$ follows from Lemma 2, Chebyshev's inequality, and property (3)
\begin{align*}
\text{I}&\leq |\Omega|+|E^*|\\
&\lesssim |\Omega|+|E|\\
&\leq \frac{1}{\lambda}\|f\|_{L^1(\mathbb{R}^n)}+\frac{1}{\lambda}\sum_{i=1}^{N}\|b_i\|_{L^1(\mathbb{R}^n)}\\
&\lesssim \frac{1}{\lambda}\|f\|_{L^1(\mathbb{R}^n)}.
\end{align*}

For $\text{II}$, use Chebyshev's inequality and Lemma 1, which applies since 
$$
    \text{supp}(b_{i}-\lambda\mathbbm{1}_{E_{i}}) \subseteq Q_{i}\cup Q(c_{i},r_{i}),\,\,\,\, \int_{\mathbb{R}^n} b_{i}(x)-\lambda\mathbbm{1}_{E_{i}}(x)dx = 0,\,\,\text{and}$$ $$2\sqrt{n}Q_{i}\cup Q(c_{i},2\sqrt{n}r_{i})\subseteq \Omega\cup E^*,
$$ 
to estimate
\begin{align*}
\text{II}&\lesssim \frac{1}{\lambda}\int_{\mathbb{R}^n\setminus(\Omega\cup E^*)}|T(b^N-\lambda\mathbbm{1}_E)(x)|dx\\
&\leq \frac{1}{\lambda}\sum_{i=1}^{N}\int_{\mathbb{R}^n\setminus (\Omega\cup E^*)}|T(b_{i}-\lambda\mathbbm{1}_{E_{i}})(x)|dx\\\
&\lesssim \frac{1}{\lambda}\sum_{i=1}^{N}\|b_{i}-\lambda\mathbbm{1}_{E_{i}}\|_{L^1(\mathbb{R}^n)}.
\end{align*}
Using the triangle inequality and property (3), we have
$$
    \text{II}\lesssim \frac{1}{\lambda}\sum_{i=1}^{N}\left(\|b_i\|_{L^1(\mathbb{R}^n)}+\|\lambda\mathbbm{1}_{E_i}\|_{L^1(\mathbb{R}^n)}\right)
    \lesssim \frac{1}{\lambda}\sum_{i=1}^{N}\|b_i\|_{L^1(\mathbb{R}^n)}
    \leq \frac{1}{\lambda}\|f\|_{L^1(\mathbb{R}^n)}.
$$

To control $\text{III}$, use Chebyshev's inequality, the boundedness of $T$ on $L^2(\mathbb{R}^n)$, and the fact that $|E|\leq \frac{1}{\lambda}\|f\|_{L^1(\mathbb{R}^n)}$ to estimate
$$
    \text{III}\lesssim \int_{\mathbb{R}^n}\left|T(\mathbbm{1}_{E})(x)\right|^{2}dx\lesssim |E|\leq  \frac{1}{\lambda}\|f\|_{L^1(\mathbb{R}^n)}.
$$

Putting all estimates together, we get 
$$
    |\{|Tf|>\lambda\}|\lesssim \frac{1}{\lambda}\|f\|_{L^1(\mathbb{R}^n)}.
$$
\end{proof}

%%%%%%%%%%%%%%%%%%%%%%%%%%%%%%%%%%%%%%%%%%%%%%%%%%%%%%%%%%%%%%%%%%%%%%%%%%

\section{Weighted Estimate}
The main difficulty in adapting the proof of Section 3 to the weighted setting is controlling the term with the ``good'' function. The following celebrated theorem is used to handle this term (see \cite{H2012}).
\begin{theorem}{3}
If $1<p<\infty$ and $w \in A_p$, then $T$ is bounded on $L^p(w)$ and 
$$
    \|T\|_{L^p(w)\rightarrow L^p(w)} \lesssim pp'[w]_{A_p}^{\text{max}\left\{1,\frac{1}{p-1}\right\}}.
$$
\end{theorem}

\begin{theorem}{2}
If $1< p<\infty$ and $w\in A_p$, then 
$$
    \|T(fw)w^{-1}\|_{L^{1,\infty}(w)} \lesssim [w]_{A_p}\max\{p,\log(e+[w]_{A_p})\}\|f\|_{L^1(w)}
$$
for all $f \in L^1(w)$.
\end{theorem}

%\begin{theorem}{3}
%\label{theorem3}
%If $1<p<\infty$, $w \in A_p$, and $\nu \in \mathcal{M}(\mathbb{R}^n)$ is of the form $\nu = \sum_{j=1}^{N} a_{j}\delta_{x_{j}}$ where $x_{j} \in \mathbb{R}^n$ and $a_{j}\in\mathbb{R}$, then for any $t>0$ $$w(\{|T\nu|w^{-1}>t\}|\leq \frac{B_4}{t}\|\nu\|$$ where $B_4$ depends on $K$, $[w]_{A_p}$, and $n$. 
%\end{theorem}

\begin{proof}[Proof of Theorem 2]
Let $\lambda>0$ be given. We wish to show 
$$
    w(\{|T(fw)|w^{-1}>\lambda\})\lesssim [w]_{A_p}\max\{p,\log(e+[w]_{A_p})\}\frac{1}{\lambda}\|f\|_{L^1(w)}.
$$ 
%Assume that $\int_{\mathbb{R}^n}w(x)dx > t^{-1}\|f\|_{L^1(w)}$ (otherwise there is nothing to prove). 
By density, we may assume $f$ is a nonnegative continuous function with compact support. Set 
$$
    \Omega:=\left\{f>\lambda\right\}.
$$ 
Apply a Whitney decomposition to write 
$$
    \Omega=\bigcup_{i=1}^{\infty}Q_i,
$$ 
a disjoint union of dyadic cubes where 
$$
    2\text{diam}(Q_i)\leq d(Q_i,\mathbb{R}^n\setminus \Omega)\leq 8\text{diam}(Q_i).
$$ 
Put 
$$
    g:=f\mathbbm{1}_{\mathbb{R}^n\setminus \Omega}, \quad\quad b:=f\mathbbm{1}_{\Omega}, \quad\quad \text{and} \quad\quad b_i:=f\mathbbm{1}_{Q_i}.
$$ 
Then 
$$
f=g+b=g+\sum_{i=1}^{\infty}b_i,
$$ 
where
\begin{itemize}
    \item[(1)] $\|g\|_{L^{\infty}(\mathbb{R}^n)}\leq \lambda$ and $\|g\|_{L^1(w)} \leq \|f\|_{L^1(w)}$,
    \item[(2)] the $b_i$ are supported on pairwise disjoint cubes $Q_i$, where $\sum_{i=1}^{\infty}w(Q_i) \leq \frac{1}{\lambda}\|f\|_{L^1(w)}$, and 
    \item[(3)] $\|b\|_{L^1(w)}\leq \|f\|_{L^1(w)}$.
\end{itemize}
Then 
$$
    w(\{|T(fw)|w^{-1}>\lambda\})\leq w\left(\left\{|T(gw)|w^{-1}>\frac{\lambda}{2}\right\}\right)+w\left(\left\{|T(bw)|w^{-1}>\frac{\lambda}{2}\right\}\right).
$$
To control the first term, let $r>p$ be a constant to be chosen later (we will actually choose $r$ so that $r>2$ as well). Then $w \in A_r$, $[w]_{A_r}\leq [w]_{A_p}$, $w^{1-r'}\in A_{r'}$, and $\left[w^{1-r'}\right]_{A_{r'}}=[w]_{A_r}^{r'-1}$. Use Chebyshev's inequality, Theorem 3, property (1), and the facts listed above to estimate
\begin{align*}
w\left(\left\{|T(gw)|w^{-1}>\frac{\lambda}{2}\right\}\right) &\lesssim \frac{1}{\lambda^{r'}}\int_{\mathbb{R}^n}|T(gw)(x)|^{r'}w(x)^{1-r'}dx\\
&\lesssim \left(rr'\left[w^{1-r'}\right]_{A_{r'}}^{\max\left\{1,\frac{1}{r'-1}\right\}}\right)^{r'}\frac{1}{\lambda^{r'}}\int_{\mathbb{R}^n}|g(x)|^{r'}w(x)dx\\
&\lesssim r^{r'}\left[w\right]_{A_r}^{r'}\frac{1}{\lambda}\int_{\mathbb{R}^n}|g(x)|w(x)dx\\
&\leq r^{r'}\left[w\right]_{A_p}^{r'}\frac{1}{\lambda}\|f\|_{L^1(w)}.\\
\end{align*}

We next address the factors $r^{r'}$ and $[w]_{A_p}^{r'}$. First consider $r^{r'}$. Let $h(x)=\frac{1}{x}(1+x)^{1+\frac{1}{x}}$. Note that $h(1)=4$ and that $h'(x)=\frac{-1}{x^3}(1+x)^{1+\frac{1}{x}}\log(1+x) \leq 0$ for all $x \in [1,\infty)$. Thus $h(x)\leq 4$ for all $x \in [1,\infty)$. In particular, letting 
$$
    r=1+\max\{p,\log(e+[w]_{A_p})\} > 2
$$ 
and computing 
$$
    r'=1+\frac{1}{\max\{p,\log(e+[w]_{A_p})\}} < 2,
$$ 
we have 
$$
    \frac{r^{r'}}{\max\{p,\log(e+[w]_{A_p})\}}=h(\max\{p,\log(e+[w]_{A_p})\})\leq 4.
$$ 
Thus 
$$
    r^{r'}\leq 4\max\{p,\log(e+[w]_{A_p})\}.
$$ 

Now consider $[w]_{A_p}^{r'}$. Set $k(x)=x^{\frac{1}{\log(e+x)}}$. Notice that $k(1)=1$, $\displaystyle \lim_{x\rightarrow\infty}k(x) = e$, and $k'(x)=x^{\frac{1}{\log(e + x)}} \left(\frac{1}{x\log(e + x)} - \frac{\log(x)}{(e + x) (\log(e + x))^2}\right) \ge 0$ for all $x\in[1,\infty)$. Thus $1\leq k(x) \leq e$ for all $x \in [1,\infty)$. In particular, 
$$
    [w]_{A_p}^{r'-1}=[w]_{A_p}^{\frac{1}{\max\{p,\log(e+[w]_{A_p})\}}}\leq[w]_{A_p}^{\frac{1}{\log(e+[w]_{A_p})}} =k([w]_{A_p})\leq e.
$$ 
Thus 
$$
    [w]_{A_p}^{r'}\leq e[w]_{A_p}.
$$ 
Substituting this into the previous estimate yields 
$$
w\left(\left\{|T(gw)w^{-1}|>\frac{\lambda}{2}\right\}\right) \lesssim r^{r'}\left[w\right]_{A_p}^{r'}\frac{1}{\lambda}\|f\|_{L^1(w)} \lesssim [w]_{A_p}\max\{p,\log(e+[w]_{A_p})\}\frac{1}{\lambda}\|f\|_{L^1(w)}.
$$

For positive integers $N$, set $b^N:=\sum_{i=1}^Nb_i$. To control the second term, it suffices to handle $w\left(\left\{|T(b^Nw)|w^{-1}>\frac{\lambda}{2}\right\}\right)$ uniformly in $N$. Let $c_i$ denote the center of $Q_i$ and let $a_i:=\int_{Q_i}b_i(x)w(x)dx$. Set 
$$
    E_{1}:=Q(c_{1},r_{1}),
$$ 
where $r_{1}>0$ is chosen so that $w(E_{1})=\frac{a_1}{\lambda}$. In general, for $i=2, 3,\ldots,N$, set 
$$
    E_{i}:=Q(c_{i},r_{i})\setminus \bigcup_{k=1}^{i-1}E_{k},
$$ 
where $r_{i}>0$ is chosen so that $w(E_{i})=\frac{a_i}{\lambda}$. Note that such $E_i$ exist since the function $r \mapsto w(Q(x,r))$ increases to $\infty$ as $r\rightarrow \infty$, approaches $0$ as $r\rightarrow 0$, and is continuous from the right for almost every $x\in \mathbb{R}^n$. Define 
$$
    E:=\bigcup_{i=1}^{N}E_{i}=\bigcup_{i=1}^N Q(c_i,r_i) \quad\quad \text{and} \quad\quad E^*:=\bigcup_{i=1}^{\infty}Q(c_i,2\sqrt{n}r_i).
$$
Then 
$$
    w\left(\left\{|T(b^Nw)|w^{-1}>\frac{\lambda}{2}\right\}\right)\leq \text{I}+\text{II}+\text{III},
$$
where
\begin{align*}
    &\text{I}:=w(\Omega\cup E^*),\\
    &\text{II}:=w\left(\left\{x \in \mathbb{R}^n\setminus(\Omega\cup E^*):|T(b^Nw-\lambda w\mathbbm{1}_E)(x)|w(x)^{-1}>\frac{\lambda}{4}\right\}\right), \quad \text{and}\\
    &\text{III}:=w\left(\left\{|T(w\mathbbm{1}_E)|w^{-1}>\frac{1}{4}\right\}\right).
\end{align*}

The control of I follows from Lemma 2, Chebyshev's inequality, and property (3)
\begin{align*}
\text{I}&\leq w(\Omega)+w(E^*)\\
&\lesssim w(\Omega)+[w]_{A_p}w(E)\\
&\leq \frac{1}{\lambda}\|f\|_{L^1(w)}+[w]_{A_p}\frac{1}{\lambda}\sum_{i=1}^{N}\|b_i\|_{L^1(w)}\\
&\lesssim [w]_{A_p}\max\{p,\log(e+[w]_{A_p})\}\frac{1}{\lambda}\|f\|_{L^1(w)}.
\end{align*}

For $\text{II}$, use Chebyshev's inequality and Lemma 1, which applies since 
$$
    \text{supp}(b_{i}w-\lambda w\mathbbm{1}_{E_{i}}) \subseteq Q_{i}\cup Q(c_{i},r_{i}),\,\, \int_{\mathbb{R}^n} b_{i}(x)w(x)-\lambda w(x)\mathbbm{1}_{E_{i}}(x)dx = 0,\,\,\text{and}
$$
$$
    2\sqrt{n}Q_{i}\cup Q(c_{i},2\sqrt{n}r_{i}) \subseteq \Omega\cup E^*,
$$
to estimate
\begin{align*}
\text{II}&\lesssim \frac{1}{\lambda}\int_{\mathbb{R}^n\setminus(\Omega\cup E^*)}|T(b^Nw-\lambda w\mathbbm{1}_E)(x)|dx\\
&\leq \frac{1}{\lambda}\sum_{i=1}^{N}\int_{\mathbb{R}^n\setminus (\Omega\cup E^*)}|T(b_{i}w-\lambda w\mathbbm{1}_{E_{i}})(x)|dx\\
&\lesssim \frac{1}{\lambda}\sum_{i=1}^{N}\|b_{i}-\lambda\mathbbm{1}_{E_{i}}\|_{L^1(w)}.
\end{align*}
Using the triangle inequality and property (3), we have 
$$
    \text{II}\lesssim \frac{1}{\lambda}\sum_{i=1}^{N}\left(\|b_i\|_{L^1(w)}+\|\lambda\mathbbm{1}_{E_i}\|_{L^1(w)}\right) \lesssim \frac{1}{\lambda}\sum_{i=1}^N\|b_i\|_{L^1(w)}
$$
$$
    \leq [w]_{A_p}\max\{p,\log(e+[w]_{A_p})\}\frac{1}{\lambda}\|f\|_{L^1(w)}.
$$

To control III, let 
$$
    r=1+\max\{p,\log(e+[w]_{A_p})\}
$$ 
and use Chebyshev's inequality, Theorem 3, and the properties of $w$ described when bounding $w(\{|T(gw)|w^{-1}>\frac{\lambda}{2}\})$ above to estimate
$$
    \text{III}\lesssim \int_{\mathbb{R}^n}|T(w(\mathbbm{1}_{E}))|^{r'}w(x)^{1-r'}dx \lesssim \left(rr'\left[w^{1-r'}\right]_{A_{r'}}^{\max\left\{1,\frac{1}{r'-1}\right\}}\right)^{r'}w(E) \lesssim r^{r'}\left[w\right]_{A_r}^{r'}\frac{1}{\lambda}\|f\|_{L^1(w)}.
$$
As before, $r^{r'}\lesssim \max\{p,\log(e+[w]_{A_p})\}$ and $[w]_{A_p}^{r'}\lesssim [w]_{A_p}$, so
$$
    \text{III}\lesssim [w]_{A_p}\max\{p,\log(e+[w]_{A_p})\}\frac{1}{\lambda}\|f\|_{L^1(w)}.
$$

Putting all estimates together, we get 
$$
    w(\{|T(fw)|w^{-1}>\lambda\})\lesssim [w]_{A_p}\max\{p,\log(e+[w]_{A_p})\}\frac{1}{\lambda}\|f\|_{L^1(w)}.
$$
\end{proof}

%%%%%%%%%%%%%%%%%%%%%%%%%%%%%%%%%%%%%%%%%%%%%%%%%%%%%%%%%%%%%%%%%%%%%%%%%%%


\begin{bibdiv}
\begin{biblist}
\bib{CRR2018}{article}{
      author={Caldarelli, M.},
      author={Rivera-Ríos, I.~P.},
       title={A sparse approach to mixed weak type inequalities},
        date={2019},
     journal={Math. Z.},
	pages={https://doi.org/10.1007/s00209-019-02447-x}
}

\bib{CUMP2005}{article}{
title={Weighted weak-type inequalities and a conjecture of Sawyer},
author={D. Cruz-Uribe},
author={J. M. Martell},
author={C. P\'erez},
date={2005},
journal={Internat. Math. Res. Notices},
volume={30},
pages={1849--1871},
}

\bib{Grafakos1}{book}{
title={Classical Fourier analysis},
author={L. Grafakos},
publisher={Springer},
edition={Third edition},
volume={249},
address={New York},
series={Graduate Texts in Mathematics},
date={2014}
}

\bib{Grafakos2}{book}{
title={Modern Fourier analysis},
author={L. Grafakos},
publisher={Springer},
edition={Third edition},
volume={250},
address={New York},
series={Graduate Texts in Mathematics},
date={2014}
}

\bib{H2012}{article}{
title={The sharp weighted bound for general Calder\'on-Zygmund operators},
author={T. P. Hyt\"onen},
date={2012},
journal={Ann. of Math.},
volume={175},
pages={1473--1506},
}

\bib{LOP2019}{article}{
title={Proof of an extension of E. Sawyer's conjecture about weighted mixed weak-type estimates},
author={K. Li},
author={S. Ombrosi},
author={C. P\'erez},
journal={Math. Ann.},
volume={374},
date={2019},
number={1-2},
pages={907--929}
}

\bib{NTV1998}{article}{
title={Weak type estimates and Cotlar inequalities for Calder\'on-Zygmund operators on nonhomogeneous spaces},
author={F. Nazarov},
author={S. Treil},
author={A. Volberg},
volume={9},
date={1998},
journal={Internat. Math. Res. Notices},
pages={463--487},
}

\bib{OP2016}{article}{
title={Mixed weak type estimates: examples and counterexamples related to a problem of E. Sawyer},
author={S. Ombrosi},
author={C. P\'erez},
volume={145},
date={2016},
journal={Colloq. Math.},
pages={259--272},
}

\bib{OPR2016}{article}{
title={Quantitative weighted mixed weak-type inequalities for classical operators},
author={S. Ombrosi},
author={C. P\'erez},
author={J. Recchi},
volume={65},
date={2016},
journal={Indiana U. Math. J.},
pages={615--640},
}

\bib{SW2019}{article}{
title={An endpoint weak-type estimate for multilinear Calder\'on-Zygmund operators},
author={C. B. Stockdale},
author={B. D. Wick},
journal={J. Fourier Anal. Appl.},
volume={25},
date={2019},
number={5},
pages={2635--2652},
}


\end{biblist}
\end{bibdiv}
\end{document}